\documentclass{amsart}
\usepackage{amsmath, amssymb, mathabx, enumitem}
\usepackage{color}
\usepackage[colorlinks=true,linkcolor=blue,urlcolor=blue,citecolor=blue, pagebackref]{hyperref}
\usepackage[alphabetic]{amsrefs}

\usepackage[margin=0.7in]{geometry}

\newcommand{\C}{\mathbb{C}}
\newcommand{\Q}{\mathbb{Q}}
\newcommand{\Z}{\mathbb{Z}}
\newcommand{\op}{\operatorname}

\newtheorem{theorem}{Theorem}[section]

\newtheorem{remark}[theorem]{ Remark}
\newtheorem{conjecture}[theorem]{Conjecture}

\newtheorem{corollary}[theorem]{Corollary}
\newtheorem{question}[theorem]{Question}
\newtheorem{proposition}[theorem]{Proposition}

\newtheorem{definition}[theorem]{Definition}
\newtheorem{definition/lemma}[theorem]{Definition/Lemma}

\title{On the saturation conjecture for $\operatorname{Spin}(2n)$}
\author{Joshua Kiers}

\begin{document}
\maketitle

\begin{abstract} In this paper we examine the saturation conjecture on decompositions of tensor products of irreducible representations for complex semisimple algebraic groups of type $D$ (the even \emph{spin} groups: Spin$(2n)$ for $n\ge 4$ an integer), extending work done by Kumar-Kapovich-Millson on Spin(8). Our main theorem asserts that the saturation conjecture holds for Spin(10) and Spin(12): for all triples of dominants weights $\lambda,\mu,\nu$ such that $\lambda+\mu+\nu$ is in the root lattice, and for any $N>0$,
$$
\left(V(\lambda)\otimes V(\mu)\otimes V(\nu)\right)^G \ne 0
$$
if and only if 
$$
\left(V(N\lambda)\otimes V(N\mu)\otimes V(N\nu)\right)^G\ne 0,
$$
for $G=\op{Spin}(10)$ or $\op{Spin}(12)$. 
Some related results for groups of other types are listed as well. 
\end{abstract}

\section{Introduction}
In this paper we examine the saturation conjecture on decompositions of tensor products of irreducible representations for complex semisimple algebraic groups of type $D$ (the \emph{spin} groups: Spin$(2n)$ for $n\ge 4$ an integer), extending work done by Kapovich-Kumar-Millson in \cite{KKM} on Spin(8). Our main theorem is that the saturation conjecture holds for Spin(10) and Spin(12). Some related results for groups of other types are listed as well. 

The saturation conjecture can be approached by studying a certain polyhedral cone, the \emph{saturated tensor cone}, whose defining inequalities are known to be minimally parametrized by products in the cohomology ring of relevant spaces $G/P$ (see \cite{BK} and \cite{Ress}, as well as the survey \cite{Kumar}). We introduce a computationally feasible method (based on the polynomial realization of \cite{BGG}) for calculating cup products in the singular (or deformed) cohomology of any $G/P$, and we indicate some pseudocode for implementing this method on a computer in order to find the desired inequalities. This can also be used to find extremal rays of the cone from the formulas of \cite{BKiers}. 
The method lends itself to (partial) parallelization, and it was with the crucial aid of the parallel-capable supercomputer {\tt Longleaf}, maintained at the University of North Carolina, that we obtained the aforementioned results. 

\subsection{The Saturation Conjecture}

Let $G$ be a semisimple algebraic group over $\C$. Fix a Borel subgroup $B\subset G$ and maximal torus $H\subset B$. Let $W=N_G(H)/H$ be the Weyl group of $H$ in $G$. The choice of $H\subset B$ determines a root system $\Phi\subset \mathfrak{h}^*$ with base (\emph{simple roots}) $\Delta=\{\alpha_1,\hdots,\alpha_r\}\subset \Phi$, where $r=\dim H$ is the rank of $G$. The $\Z$-span of the $\alpha_i$ is called the \emph{root lattice}.

Let $\omega_1,\hdots,\omega_r\in \mathfrak{h}^*$ be the associated dominant fundamental weights. The elements of the $\mathbb{Z}_{\ge0}$-span of $\{\omega_1,\hdots,\omega_r\}$ are the \emph{dominant weights}, and to each such $\lambda$ is associated a unique, irreducible representation of $G$, denoted by $V(\lambda)$. 

It is a standard problem to determine when an irreducible component $V(\nu)$ appears in a tensor product of irreducible representations $V(\lambda)\otimes V(\mu)$. This question is, in general, hard to answer. However, the question whether $V(N\nu)$ appears in $V(N\lambda)\otimes V(N\mu)$ for some $N\ge 1$ has a well-known answer. One then may ask if and how $N$ may be controlled. Replacing $\nu$ by its dual weight, the questions may be posed symmetrically, which we now make precise. 

\begin{definition}
Define the saturated tensor cone $\mathcal{C}(G)$ to be the set of triples $\lambda_1,\lambda_2,\lambda_3$ of  dominant weights whose sum is in the root lattice and for which 
\begin{align}\label{nonz}
(V(N\lambda_1)\otimes V(N\lambda_2)\otimes V(N\lambda_3))^G\ne 0
\end{align}
for some integer $N>0$. Define the tensor cone $\mathcal{R}(G)$ to be the set of triples $\lambda_1,\lambda_2,\lambda_3$ of dominant weights satisfying (\ref{nonz}) with $N=1$ (in which case $\lambda_1+\lambda_2+\lambda_3$ must lie in the root lattice). 
\end{definition}

There exists a system of inequalities (listed below) determining this cone as a subset of $(\mathfrak{h}^*)^3$; an overview can be found in the survey paper of Kumar \cite[Section 6]{Kumar}. 

%The cone $\mathcal{C}(G)$ has a natural additive structure:

%\begin{proposition}
%Suppose $\lambda_1,\lambda_2,\lambda_3,N$ satisfy (\ref{nonz}), and suppose $\lambda_1',\lambda_2',\lambda_3',N'$ are another such tuple. Then for some integer $M>0$, 
%$$
%\Big(V(M(\lambda_1+\lambda_1'))\otimes V(M(\lambda_2+\lambda_2'))\otimes V(M(\lambda_3+\lambda_3'))\Big)^G\ne 0
%$$
%\end{proposition}
%\begin{proof}
%By the Borel-Weil theorem, $(V(N\lambda_1)\otimes V(N\lambda_2)\otimes V(N\lambda_3))^G\ne 0$ if and only if the line bundle 
%$$
%\mathcal{L}_{N\lambda_1}\boxtimes\mathcal{L}_{N\lambda_2}\boxtimes\mathcal{L}_{N\lambda_3}
%$$
%on $(G/B)^3$ possesses a nonzero $G$-invariant global section, call it $\sigma$. If $\sigma'$ is an analogous section for $\mathcal{L}_{N'\lambda_1'}\boxtimes\mathcal{L}_{N'\lambda_2'}\boxtimes\mathcal{L}_{N'\lambda_3'}$, then $\sigma^{N'}\cdot(\sigma')^N$ is a nonzero $G$-invariant global section for 
%$$
%\mathcal{L}_{NN'(\lambda_1+\lambda_1')}\boxtimes\mathcal{L}_{NN'(\lambda_2+\lambda_2')}\boxtimes\mathcal{L}_{NN'(\lambda_3+\lambda_3')}.
%$$
%The result follows with $M=NN'$. 
%\end{proof}

The cone $\mathcal{C}(G)$ has a natural additive structure via $(\lambda_1,\lambda_2,\lambda_3)+(\lambda_1',\lambda_2',\lambda_3') = (\lambda_1+\lambda_1',\lambda_2+\lambda_2',\lambda_3+\lambda_3')$ , making it a monoid with identity $(0,0,0)$. This follows from the Borel-Weil theorem, and the same argument shows that $\mathcal{R}(G)$ is a monoid as well. 
Note that, by definition, $\mathcal{R}(G) \subseteq \mathcal{C}(G)$. The saturation conjecture asks about the converse: 
\begin{conjecture}\label{satc}
For $G$ simple, simply-connected, and simply-laced, 
\begin{align*}%\label{issat}
\mathcal{R}(G)=\mathcal{C}(G).
\end{align*}
\end{conjecture}

For $G$ of type $A$, Conjecture \ref{satc} is true, as demonstrated by Knutson and Tao in \cite{KT}. Furthermore, Kapovich, Kumar, and Millson proved this conjecture for $G=\text{Spin}(8)$ (type $D_4$) \cite{KKM}. It is known that if $G$ is not of simply-laced type, \ref{satc} fails: see \cite{E}, \cite{KM}, and the discussion in \cite{Kumar}. The question is still open for types $D$ and $E$ in general. The main theorem of this paper is 
\begin{theorem}\label{Maine}
The saturation conjecture holds for 
\begin{enumerate}[label = (\alph*)]
\item $G=\text{Spin}(10)$ (type $D_5$) and 
\item $G=\text{Spin}(12)$ (type $D_6$).
\end{enumerate}
\end{theorem}

The proof of this theorem will be given in Section \ref{expl}. We generally follow the approach of \cite{KKM}: the proof reduces to finding a finite set of generators for $\mathcal{C}(G)$ and verifying that these generators each belong to $\mathcal{R}(G)$. 
For part (a), we are able to produce the defining inequalities for $\mathcal{C}(G)$ and use software to deduce a generating set from these; this is exactly the \cite{KKM} approach. For part (b), the inequalities are too many in number, and accordingly we find a (redundant) set of extremal rays for $\mathcal{C}(G)$ based on the formulas of \cite{BKiers}; we then use software to deduce the minimal generating set from these rays. 

The methods we use are related to those of Pasquier and Ressayre found in \cite{PR}, where they answer a generalized saturation question for embedded subgroups in some specific instances. 

Additionally, we list a summary of computational results - number of (irredundant) inequalities, number of Hilbert basis elements, number of extremal rays - pertaining to the saturated tensor cones of types $A$, $C$, and $D$ and of small rank. For several of these examples, such computations have already been presented in the literature, and we verify that our results agree. In principal, similar computational results could be obtained for type $B$ (dual to type $C$) and the exceptional types $E,F,G$. 

Finally, we include a discussion of certain Hilbert basis elements for $\mathcal{C}(\operatorname{Spin}(10))$ which fail to have the ``Fulton scaling property.'' It was conjectured and proven that all elements of type $A$ cones have this property, but a strictly weaker statement holds for general type. 

\subsection{Acknowledgements}

The author thanks R. Rim\'anyi and S. Kumar for helpful discussions surrounding Proposition \ref{orange}, S. Kumar for bringing Corollary \ref{linchpin} to light and for suggesting some references, and P. Belkale for suggesting changes to the manuscript. 

The author also thanks S. Sarangi and the ITS Research Computing team at UNC-CH for help using {\tt Longleaf}. 

\section{Notation and Preliminaries}
We fix the following notation:

\begin{enumerate}[label = $\bullet$]
\item $G$ is a simply-connected complex semisimple algebraic group, with Lie algebra $\mathfrak{g}$;

\item $B$ is a fixed Borel subgroup of $G$, with Lie algebra $\mathfrak{b}$;

\item $H\subset B$ is a fixed maximal torus of $G$, with Lie algebra $\mathfrak{h}$;

\item $\mathfrak{h}^*$ is the vector space dual to $\mathfrak{h}$;

\item $\Phi\subset \mathfrak{h}^*$ is the root system of $\mathfrak{h}$ in $\mathfrak{g}$;

\item $\Phi^+$ is the set of positive roots w.r.t. $\mathfrak{b}$, and $\Delta=\{\alpha_1,\hdots,\alpha_r\}\subset \Phi^+$ is the set of simple roots;

\item $\{\omega_1,\hdots,\omega_r\}\subset \mathfrak{h}^*$ is the set of dominant fundamental weights;

\item $\{x_1,\hdots,x_r\}$ is the dual basis for $\mathfrak{h}$ relative to $\Delta$;

\item $\{\alpha_1^\vee,\hdots,\alpha_r^\vee\}$ is the dual basis for $\mathfrak{h}$ relative to the fundamental weights;

\item $W=N_G(H)/H$ is the Weyl group of $G$, with longest element $w_0$;

\item $\sigma_\gamma\in W$ is the reflection across the hyperplane $\gamma=0\subset \mathfrak{h}$ for $\gamma\in \mathfrak{h}^*$;

\item $\ell(w)$ denotes the length of an element $w\in W$;

\item if $P\supset B$ is a standard parabolic subgroup with Lie algebra $\mathfrak{p}$, then $\Delta(P)$ denotes the subset of simple roots whose negatives appear in $\mathfrak{p}\subset \mathfrak{g}$;

\item $L$ denotes the Levi subgroup of $P$, and $L^{ss}$ the semisimple part of $L$;

\item $W_P\subset W$ denotes the Weyl group of $P$ (that is, $N_L(H)/H$), with longest element $w_0^P$;

\item $W^P$ is the set of minimal-length left coset representatives of $W_P$ in $W$;

\item if $w\in W$ (resp., $w\in W^P$), then define $X_w = \overline{BwB}$ (resp., $X_w^P = \overline{BwP}$), a subvariety of $G/B$ (resp., $G/P$) of dimension $\ell(w)$;

\item $\mu(X_w)$ (resp., $\mu(X_w^P)$) is the associated fundamental class in $H_{2\ell(w)}(G/B)$ (resp., $H_{2\ell(w)}(G/P)$);

\item $[X_w]$ (resp., $[X_w^P]$) is the Poincar\'e dual to $\mu(X_w)$ (resp., $\mu(X_w^P)$) and is an element in $H^{2(\dim G/B-\ell(w))}(G/B)$ (resp., $H^{2(\dim G/P-\ell(w))}(G/P)$);

\item the cup product in $H^*(G/B)$ or $H^*(G/P)$ will be denoted by $\cdot$, and the deformed cup product by $\odot_0$;

\item $\rho$ is the half-sum of positive roots, and $\rho^L$ is the half-sum of positive roots for the root system of $L$;

\item for $w\in W^P$, $\chi_w := \rho-2\rho^L+w^{-1}\rho$;

\item for $v,w\in W$, $\beta\in \Phi$, $v\xrightarrow{\beta}w$ means $w = \sigma_\beta v$ and $\ell(w) = \ell(v)+1$; 

\item if, furthermore, both $v,w\in W^P$, we write $v\xrightarrow{\beta}w\in W^P$. 
\end{enumerate}

\subsection{Inequalities for the Tensor Cone}\label{all}
Let $\lambda_1,\lambda_2,\lambda_3$ be dominant weights whose sum is in the root lattice. Then by \cite{BK}, 
$$
(V(N\lambda_1)\otimes V(N\lambda_2)\otimes V(N\lambda_3))^G\ne 0 
$$
for some integer $N>0$ (i.e., $(\lambda_1,\lambda_2,\lambda_3)\in \mathcal{C}(G)$) if and only if for every maximal standard parabolic $P=P_i\subset G$ and every triple $(w_1,w_2,w_3)\in (W^P)^3$ satisfying 
\begin{align}\label{product}
[X_{w_1}^P]\odot_0[X_{w_2}^P]\odot_0[X_{w_3}^P]=[X_e^P],
\end{align}
the inequality 
\begin{align}\label{ineqs}
\left(\sum_{j=1}^3w_j^{-1}\lambda_j\right)(x_i)\le 0
\end{align}
holds. Note that this elucidates the monoidal structure of $\mathcal{C}(G)$, since the inequalities (\ref{ineqs}) are linear. 

By the definition of the deformed product $\odot_0$, a triple $(w_1,w_2,w_3)\in (W^P)^3$ satisfies (\ref{product}) if and only if it satisfies
\begin{align}\label{lightning}
[X_{w_1}^P]\cdot[X_{w_2}^P]\cdot[X_{w_3}^P]=[X_e^P]\text{ and }\left(\chi_{w_1}+\chi_{w_2}+\chi_{w_3}-\chi_{1}\right)(x_i)=0;
\end{align}
where $\chi_w$ is as defined above. 

Fixing a basis for $\mathfrak{h}^*$ will allow the inequalities to be understood by a computer; see Section \ref{CALC} for discussion of the inequalities. The computer software {\tt Normaliz} \cite{Normaliz}, among others, is capable of reporting various characteristics of the cone $\mathcal{C}(G)$ given these defining inequalities. 

\subsection{Facets of the Tensor Cone}

It was demonstrated by Ressayre in \cite{Ress} that the inequalities (\ref{ineqs}) are irredundant. Therefore, the subcones
$$
\mathcal{F}(\vec w,P) = \left\{\vec\lambda \in \mathcal{C}(G) \mid \left(\sum_{j=1}^3w_j^{-1}\lambda_j\right)(x_i)= 0\right\}
$$
given by $(w_1,w_2,w_3),P$ satisfying (\ref{product}) form regular facets of $\mathcal{C}(G)$; i.e., they are codimension 1 faces (hence facets) and not contained in any dominant chamber wall $\{\lambda_i(\alpha_j^\vee)=0\}$ (hence regular). The only other facets of $\mathcal{C}(G)$ are those coming from the dominant criterion: each $\lambda_i$ must be a dominant weight; i.e., $\lambda_i(\alpha_j^\vee)\in \Z_{\ge0}$ for each $j$.

\subsection{Extremal rays of $\mathcal{C}(G)$}\label{taketwo}

The $D_5$ calculation can be carried out once the inequalities have been generated. However, for the $D_6$ calculation we will need to generate the extremal rays of $\mathcal{C}(G)$ according to the formulas given in \cite{BKiers}, which we recall here. 

Every extremal ray of $\mathcal{C}(G)$ lies on a facet $\mathcal{F}(\vec w,P)$ (see \cite{BKiers}*{Lemma 5.4}). Here we list the rays lying on a given $\mathcal{F}(\vec w,P)$. 

\subsubsection{Type I rays of $\mathcal{F}(\vec w,P)$}

Let $j,\ell$ be such that $v\xrightarrow{\alpha_\ell}w_j$ (this implies $v\in W^P$ as well). Then there is a ray $r_{j,\ell}$ defined as follows. Set $u_i = w_i$ for all $i$ except for $u_j = v$. Writing
\begin{align}\label{tI}
r_{j,\ell} = \left(\sum_{k=1}^rc_k^{(1)}\omega_k,\sum_{k=1}^rc_k^{(2)}\omega_k,\sum_{k=1}^rc_k^{(3)}\omega_k\right)
\end{align}
in the basis of fundamental weights, $c_k^{(i)}$ is the number $c$ in 
\begin{align}\label{product2}
[X_{\hat u_1}^P]\cdot[X_{\hat u_2}^P]\cdot[X_{\hat u_3}^P]=c[X_e^P]
\end{align}
if $u_i\xrightarrow{\alpha_k}\sigma_{\alpha_k}u_i\in W^P$, and $0$ otherwise, where $\hat u_m = u_m$ for all $m$ except for $\hat u_i = \sigma_{\alpha_k}u_i$. 

By \cite{BKiers}*{Theorem 1.6}, $r_{j,\ell}$ gives an extremal ray of $\mathcal{C}(G)$ on the face $\mathcal{F}(\vec w,P)$, and rays arising in this manner are called ``Type I'' on the facet (indeed, they could be Type II on an adjacent facet; there is not a firm dichotomy).

\subsubsection{Type II rays of $\mathcal{F}(\vec w,P)$} 

The Type I rays of $\mathcal{F}(\vec w,P)$ do not give all extremal rays on that facet. There is a subcone $\mathcal{F}_2\subset \mathcal{F}(\vec w,P)$ whose extremal rays  - call them ``Type II'' - are the remaining extremal rays of $\mathcal{F}(\vec w,P)$. As explained in \cite{BKiers}*{Section 9}, there is a surjection of cones 

$$
\operatorname{Ind}: \mathcal{C}(L^{ss})\twoheadrightarrow \mathcal{F}_2
$$
given by the following formula. For a triple of weights $\mu_1,\mu_2,\mu_3\in \mathcal{C}(L^{ss})\subset \mathfrak{h}_{L^{ss}}^*$, extend each $\mu_i$ to an element of $\mathfrak{h}^*$, requiring that $\mu_i(x_j)=0$ for the $j$ such that $\alpha_j\not\in \Delta(P)$. With slight abuse of notation, we use $\mu_i$ to refer to the so-obtained elements of $\mathfrak{h}^*$. Then 

\begin{align}\label{tII}
\operatorname{Ind}(\mu_1,\mu_2,\mu_3) = (w_1\mu_1,w_2\mu_2,w_3\mu_3) - \sum_{j=1}^s \sum_{\ell}^{'}w_j\mu_j(\alpha_\ell^\vee)\cdot r_{j,\ell},
\end{align}
where the second sum is over those $\ell$ satisfying $v\xrightarrow{\alpha_\ell}w_j$ (i.e., those for which there exists a ray $r_{j,\ell}$). 

Although extremal rays of $\mathcal{C}(L^{ss})$ may be sent to a non-extremal ray of $\mathcal{F}_2$ or even to $0$ under $\operatorname{Ind}$, every (``Type II'') extremal ray of $\mathcal{F}_2$ is the image of an extremal ray of $\mathcal{C}(L^{ss})$.

\section{Reduction to smaller groups}

Using either the inequalities of Section \ref{all} or the extremal rays of Section \ref{taketwo}, standard techniques and the aid of a computer yield the Hilbert basis of the cone $\mathcal{C}(G)$. The Hilbert basis is the unique minimal set of monoid generators (over $\Z$) of the cone $\mathcal{C}(G)$. 

Once the Hilbert basis is obtained, the question remains whether each basis element is in fact a member of $\mathcal{R}(G)$. Greatly reducing that burden is the following result of Roth (see \cite{Roth}). 

\begin{theorem}\label{roreduct}
Suppose $(w_1,w_2,w_3),P$ satisfy (\ref{product}). Let $(\lambda_1,\lambda_2,\lambda_3)\in \mathcal{F}(\vec w,P)$. Define $L^{ss}$ to be the semisimple part of $P$, and set $\overline{\lambda_j} = w_j^{-1}\lambda_j, j=1,2,3$. Then there exists an isomorphism
\begin{align*}%\label{crucial}
\left(V(\lambda_1)\otimes V(\lambda_2)\otimes V(\lambda_3)\right)^G \simeq
\left(V(\overline{\lambda_1})\otimes V(\overline{\lambda_2})\otimes V(\overline{\lambda_3})\right)^{L^{ss}}.
\end{align*}

\end{theorem}

Roth's original theorem is more general; there are lower-dimensional regular faces $\mathcal{F}(\vec w,P)$ of $\mathcal{C}(G)$ coming from non-maximal parabolics $P$, and the same theorem holds there, too. The following application of Roth's theorem was brought to the author's attention by S. Kumar; details were discussed by the author and P. Belkale: 

\begin{corollary}\label{linchpin}
Suppose $\vec\lambda\in \mathcal{C}(G)$ lies on a regular face $\mathcal{F}(\vec w,P)$, and suppose the saturation conjecture holds for $L^{ss}$. Then $\vec\lambda\in \mathcal{R}(G)$. 
\end{corollary}

\begin{proof}
Because $\vec\lambda\in \mathcal{C}(G)$, there exists $N>0$ so that 
$$
(V(N\lambda_1)\otimes V(N\lambda_2)\otimes V(N\lambda_3))^G\ne 0.
$$
Since $(N\lambda_1,N\lambda_2,N\lambda_3)$ is also in $\mathcal{F}(\vec w,P)$, and since $\overline{N\lambda_j} = N\overline{\lambda_j}$, Theorem \ref{roreduct} gives
$$
(V(N\overline{\lambda_1})\otimes V(N\overline{\lambda_2})\otimes V(N\overline{\lambda_3}))^{L^{ss}}\ne 0.
$$
 
%%citeme
 
Recall that, for any $w\in W$ and dominant weight $\lambda$, $\lambda-w\lambda$ is in the root lattice (see \cite{Hu}). Therefore 
$$
\sum \lambda_j - \sum w_j^{-1}\lambda_j = \sum (\lambda_j-w_j^{-1}\lambda_j)
$$
is in the root lattice. Since $\lambda_1+\lambda_2+\lambda_3$ is in the root lattice, $\overline{\lambda_1}+\overline{\lambda_2}+\overline{\lambda_3}$ is in the root lattice for $G$. Furthermore, the equation defining $\mathcal{F}(\vec w,P)$ implies that $\overline{\lambda_1}+\overline{\lambda_2}+\overline{\lambda_3}$ is indeed in the root lattice for $L^{ss}$. 
So $(\overline{\lambda_1},\overline{\lambda_2},\overline{\lambda_3})\in \mathcal{C}(L^{ss})$ and therefore also lies in $\mathcal{R}(L^{ss})$. Thus 
$$
(V(\overline{\lambda_1})\otimes V(\overline{\lambda_2})\otimes V(\overline{\lambda_3}))^{L^{ss}}\ne 0,
$$
and the result follows from another application of Theorem \ref{roreduct}. 
\end{proof}

\section{Calculation methods}\label{CALC}

We discuss here the methods used to generate the inequalities (\ref{ineqs}) and rays according to (\ref{tI}) and (\ref{tII}). A ring isomorphism $H^*(G/B)\simeq R/J$ is described thanks to \cite{BGG}, and we explain a method for our specific computations in $R/J$ using only arithmetic. We then indicate how a computer might use these calculations to explicitly parametrize the desired inequalities. 

\subsection{Polynomial realization of $H^*(G/P)$}

The ring $H^*(G/P)$ may be described by polynomials, cf. \cite{BGG}.
There is a ring homomorphism $\pi^*:H^*(G/P)\to H^*(G/B)$ induced by the standard projection $\pi:G/B\to G/P$. Because of the Bruhat decomposition, $\pi^*$ is an injection, and it satisfies
$$
\pi^*\left([X_{w_0ww_0^P}^P]\right) = [X_{w_0w}]
$$
for any $w\in W^P$. Furthermore, there is a ring isomorphism 
\begin{align}\label{isom}
R/J \simeq H^*(G/B;\Q),
\end{align}
where $R = \text{Sym}^\bullet(\mathfrak{h}^*) = \Q[\alpha_i]$ and $J$ is the ideal generated by all $W$-invariant polynomials with no constant term. Since $H^*(G/B)$ is a free $\Z$-module, $H^*(G/B;\Q) = H^*(G/B)\otimes \Q$ and no products in $H^*(G/B)$ are trivialized in $H^*(G/B;\Q)$; that is, we are free to calculate coefficients of products in $H^*(G/B;\Q) = R/J$ and interpret them as coefficients of the corresponding products in $H^*(G/B)$. 

\subsection{Polynomials and integration}

For any $\gamma\in \Phi$, define $A_\gamma:R\to R$ by 
$$
A_\gamma(f) = \frac{f-\sigma_\gamma f}{\gamma},
$$
where $\sigma_\gamma\in W$ is the reflection across the hyperplane perpendicular to $\gamma$. As shown in \cite{BGG}, $A_\gamma$ is well-defined and, if $w = \sigma_{\gamma_1}\cdots\sigma_{\gamma_t}$ is a minimal length decomposition of $w$, 
$$
A_w:=A_{\gamma_1}\circ\cdots\circ A_{\gamma_t}
$$
does not depend on the choice of minimal decomposition. The operators $A_w$ descend to well-defined operators on $R/J$, and one easily checks that $A_\gamma^2=0$. These operators generate a good basis of $R/J$:
\begin{definition}
Define $\tilde P_{w_0} = \frac{1}{|W|} \prod_{\alpha\in \Phi^+}\alpha\in R$, and define 
$$
\tilde P_w:=A_{w^{-1}w_0}P_{w_0}
$$
for all other $w\in W$. Let $P_w$ denote the image of $\tilde P_w$ in $R/J$. 
\end{definition}

We record various properties of the $P_w$:
\begin{proposition}\label{oldie}
\begin{enumerate}[label=(\alph*)]
\item The collection $\{P_w\}$ forms a $\Q$-basis for $R/J$.
\item Under the isomorphism (\ref{isom}), $P_w\mapsto [X_{w_0w}]$.
\item Each $\tilde P_w$ is homogeneous of degree $\ell(w)$.
\item Any $f\in R$ may be written as 
$$
f=\sum\tilde P_w f_w,
$$
where each $f_w$ is $W$-invariant. 
\item For any $w\in W$, $P_wP_{ww_0}=P_{w_0}$.
\end{enumerate}
\end{proposition}

Now define a linear functional $\Psi:R\to \Q$ as follows:
$$
\Psi(f) = \frac{1}{\prod_{\alpha\in \Phi^+}\alpha} \sum_{\sigma\in W}(-1)^{\ell(\sigma)}\sigma(f)\Bigg|_{0},
$$
where $\big|_0$ means evaluation of a polynomial in $ \text{Sym}^\bullet(\mathfrak{h}^*)$ at $0\in \mathfrak{h}$. It is known that the linear operators on $R$
$$
A_{w_0} ~~~~\text{ and }~~~~ \frac{1}{\prod_{\alpha\in \Phi^+}\alpha} \sum_{\sigma\in W}(-1)^{\ell(\sigma)}\sigma(\cdot)
$$
coincide (see, for example, \cite{Lasc}). The following properties of $\Psi$ follow readily. 

\begin{proposition}\label{orange}
The map $\Psi$ is well-defined, and $\Psi(f)$ is the $P_{w_0}$-coefficient of $\bar f\in R/J$. If $f\in R$ has degree $\le \deg \tilde P_{w_0}$, evaluation at $0$ may be replaced by evaluation at any element of $\mathfrak{h}_\Q$. 
\end{proposition}
\begin{proof}
$\Psi$ is well-defined since $A_{w_0}$ is. For any $f\in R$, write $f = \sum \tilde P_w f_w$ as in Proposition \ref{oldie}(d). Since $A_{w_0}\tilde P_w=0$ for any $w\ne w_0$, $A_{w_0}f \big|_0 = f_{w_0}(0)$, which is the $P_{w_0}$-coefficient of $\bar f$ in $R/J$. 

If $f$ has degree $\le \deg(\tilde P_{w_0})$, then $A_{w_0}f$ is a constant by degree considerations and evaluation at $0$ may be replaced with evaluation at any element of $\mathfrak{h}_\Q$. 

\end{proof}

\begin{remark}See the appendix for another, more direct, derivation of the above proposition. \end{remark}
Since $\Psi$ vanishes on $J$, we write $\Psi$ again for the induced operator $R/J\to \Q$. The following corollaries explain that $\Psi$ may be viewed as integration of forms on $G/B$ and how this is useful for products in $H^*(G/P)$. 

\begin{corollary}
Viewed as a linear functional $H^*(G/B;\Q)\to \Q$, $\Psi$ is the same as capping with the fundamental class $\mu(X_e)\in H_*(G/B;\Q)$. 
\end{corollary}

\begin{corollary}
Given $w_1,w_2,w_3\in W$ such that $\ell(w_1)+\ell(w_2)+\ell(w_3)=\ell(w_0)$, the number $c$ in 
$$
[X_{w_0w_1}]\cdot[X_{w_0w_2}]\cdot[X_{w_0w_3}] = c[X_e]
$$
may be computed as 
$$
c = \mu(X_e)\cap c[X_e] = \Psi(\tilde P_{w_1}\tilde P_{w_2}\tilde P_{w_3});
$$
furthermore, since $\tilde P_{w_1}\tilde P_{w_2}\tilde P_{w_3}$ has degree no greater than $\ell(w_0)$, 
$$
c = \frac{1}{|W|\tilde P_{w_0}(h)}\sum_{\sigma\in W} (-1)^{\ell(\sigma)}\tilde P_{w_1}(\sigma^{-1}h) \tilde P_{w_2}(\sigma^{-1}h) \tilde P_{w_3}(\sigma^{-1}h)
$$
for any $h\in \mathfrak{h}_\Q$. 
\end{corollary}

\begin{corollary}\label{useful}
Given $w_1,w_2,w_3\in W^P$ such that $\ell(w_1)+\ell(w_2)+\ell(w_3)=\ell(w_0^P)$, the number $c$ in 
$$
[X_{w_0w_1w_0^P}^P]\cdot[X_{w_0w_2w_0^P}^P]\cdot[X_{w_0w_3w_0^P}^P] = c[X_{e}^P]
$$
is the same as $c$ in (under $\pi^*$)
$$
[X_{w_0w_1}]\cdot[X_{w_0w_2}]\cdot[X_{w_0w_3}] = c[X_{w_0^P}],
$$
which is the number $c$ in 
$$
[X_{w_0w_1}]\cdot[X_{w_0w_2}]\cdot[X_{w_0w_3}]\cdot [X_{w_0w_0^P}] = c[X_{w_0^P}]\cdot [X_{w_0w_0^P}] = c[X_e];
$$
therefore 
$$
c = \frac{1}{|W|\tilde P_{w_0}(h)}\sum_{\sigma\in W} (-1)^{\ell(\sigma)}\tilde P_{w_1}(\sigma^{-1}h) \tilde P_{w_2}(\sigma^{-1}h) \tilde P_{w_3}(\sigma^{-1}h) \tilde P_{w_0^P}(\sigma^{-1}h)
$$
for any $h\in \mathfrak{h}_\Q$ such that $\tilde P_{w_0}(h)\ne 0$. 
\end{corollary}

This last corollary is what we use to calculate the coefficient $c$ in cohomology products, via the method below. 

\subsection{Pseudocode for products and inequalities}

Given a computer package with sufficient knowledge of root systems and their associated Weyl groups (as available through {\tt Sage} \cite{sagemath}, for example), one can deduce the defining inequalities (\ref{ineqs}) for $\mathcal{C}(G)$ once one knows the set of all triples $(w_1,w_2,w_3)\in (W^P)^3$ satisfying (\ref{product}) (or, equivalently, (\ref{lightning})), for all maximal standard $P$. The question of computing the cup product in (\ref{lightning}) is reduced to computing a sum of polynomials evaluated on a fixed vector $h\in \mathfrak{h}_\Q$ as in Corollary \ref{useful}. 

Some math software (such as {\tt Sage}) is capable of polynomial manipulation and simplification. However, the following pseudocode illustrates that the need for polynomial handling can be replaced with rudimentary data storage and arithmetic. \\

{\tt 
dict = \{\};  {\color{blue}\# this dictionary will hold values of \textnormal{$\tilde P_{w}$} for each \textnormal{$w\in W$}.}

weylgroup.sort();  {\color{blue}\# list the elements of \textnormal{$W$} in order of decreasing length.}

h = rho;  {\color{blue}\# the half-sum of positive roots (or set h to anything not in the root hyperplanes)}

val = 1;

for a in positiveroots:

\hspace{0.2in}val = val*a(h);

val = val/len(weylgroup);

dict[weylgroup[0]] = [val*(-1)\string^length(s) for s in weylgroup];  {\color{blue}\# list of values for \textnormal{$\tilde P_{w_0}$}}

for w in weylgroup[1:]: {\color{blue} \# all except the longest element}

\hspace{0.2in}{\color{blue} \# find simple reflection \textnormal{$s_\gamma$} so that $\ell(ws_\gamma)>\ell(w)$. }

\hspace{0.2in}{\color{blue} \# then use $A_\gamma\tilde P_{ws_\gamma} = \tilde P_w$ to compute the values $\tilde P_w(th), t\in W$. }

\hspace{0.2in}for i in [1,...,rank]:

\hspace{0.4in}s = simplereflections[i];

\hspace{0.4in}if length(w*s) > length(w):

\hspace{0.6in}exit for

\hspace{0.2in}listofvals = [];

\hspace{0.2in}for t in weylgroup:

\hspace{0.4in}{\color{blue} \# here j(t) returns the index so that weylgroup[j(t)] = t.}

\hspace{0.4in}listofvals += [(dict[w*s][j(t)] - dict[w*s][j(s*t)])/simpleroots[i](h)];   

\hspace{0.2in}dict[w] = listofvals;\\\\

}

The dictionary {\tt dict} now contains a list for each $w\in W$; that list is the set of values $\tilde P_w(t.h)$ where $t$ ranges over all elements of $W$. The above algorithm can be quasi-parallelized: subsequent dictionary entries need only a single previous entry to be populated before going forward. Integration is now straightforward:\\

{\tt  
def integrate(w1,w2,w3,i):  {\color{blue} \# here i is such that $P=P_i$.}

\hspace{0.2in}{\color{blue} \# Below w0(i) is the longest element of $W_P$.}

\hspace{0.2in}sum = 0;

\hspace{0.2in}for j in [0,...,len(weylgroup)-1]:

\hspace{0.4in}sum += (-1)\string^length(weylgroup[j])*dict[w1][j]*dict[w2][j]*dict[w3][j]*dict[w0(i)][j];

\hspace{0.2in}return sum/(len(weylgroup)*dict[w0][0]);\\\\

}

The algorithm for generating the inequalities coming from (\ref{ineqs}) is also straightforward:\\

{\tt 
ineqs = [];

for i in [1,...,rank]:

\hspace{0.2in}for w1,w2,w3 in weylgroup:  {\color{blue} \# such that $\ell(w_1)+\ell(w_2)+\ell(w_3)+\ell(w_0^P) = \ell(w_0)$}

\hspace{0.4in}c = integrate(w1,w2,w3,i);

\hspace{0.4in}if c == 1:

\hspace{0.6in}if (chi(w1)+chi(w2)+chi(w3)-chi(1))(x(i)) == 0:

\hspace{0.8in}(v1,v2,v3) = (w0*w1*w0(i),w0*w2*w0(i),w0*w3*w0(i));

\hspace{0.8in}ineqs += [v1*x(i)+v2*x(i)+v3*x(i)];  {\color{blue} \# express the $v_jx_i$ in the $\omega_k$s, then concatenate. }

}

\subsection{Method for rays}

The usage of formulas (\ref{tI}) and (\ref{tII}) is straightforward to implement. The intersection numbers $c$ in (\ref{product2}) are calculated exactly by the integration described above. These numbers give the type I rays. The type II rays do not require the product method; they rely on complete knowledge of the type I rays, knowledge of rays from the Levi, and ability to calculate the Weyl group action and pairing with $\mathfrak{h}$ of elements of $\mathfrak{h}^*$. All this can also be accomplished in {\tt Sage}. 

\section{Proof of Theorem \ref{Maine}}\label{expl}

\subsection{Proof of part (a)}

Via computer (code written in {\tt Sage 8.0} \cite{sagemath}), we obtained the following conditions governing the cone $\mathcal{C}(\text{Spin}(10))$:

\begin{enumerate}[label=$\bullet$]
\item $1967$ inequalities coming from (\ref{ineqs}), using the algorithm described above
\item $15$ chamber inequalities
\item $2$ equalities (for ensuring the sum is in the root lattice)
\end{enumerate}

Submitting these inequalities to the freely available software {\tt Normaliz} \cite{Normaliz}, and with the aid of supercomputer {\tt Longleaf}, we found that the Hilbert basis for $\mathcal{C}(\text{Spin}(10))$ consists of $505$ elements, all of which lie on some regular facet. The computation failed to complete on a regular computer but successfully finished in $4$ hours on the supercomputer. 

The possible $L^{ss}$ subgroups arising from regular facets are of the following types: 
$D_4$, $A_1\times A_3$, $A_2\times A_1\times A_1$, and $A_4$.  It is known that the saturation conjecture holds for each of these (see \cite{KKM}, \cite{KT}), so by Corollary \ref{linchpin}, each Hilbert basis element $\vec\lambda$ is in $\mathcal{R}(G)$. This shows $\mathcal{C}(G)\subseteq \mathcal{R}(G)$, and the result follows. 

\begin{remark}
We also checked explicitly (using the freely available software {\tt LiE} \cite{LiE}) that each Hilbert basis element $\vec\lambda \in \mathcal{R}(G)$. 
\end{remark}

\subsection{Proof of part (b)}

Via computer (code written in {\tt Sage 8.0} \cite{sagemath}), we used formulas (\ref{tI}) and (\ref{tII}) to produce the extremal rays of $\mathcal{C}(\operatorname{Spin}(12))$. (Note: this still requires determination of all $w_1,w_2,w_3,P$ satisfying (\ref{product}), which is as computable as knowledge of the inequalities.) Because each ray lies on multiple regular facets, and because formula (\ref{tII}) possibly produces non-extremal rays, the obtained set of elements of $\mathcal{C}(\operatorname{Spin}(12))$ was not minimal in generating $\mathcal{C}(\operatorname{Spin}(12))$ over $\Q_{\ge 0}$. Indeed, we obtained $105343$ such elements. We then used the freely available software {\tt Normaliz} \cite{Normaliz}, with the aid of supercomputer {\tt Longleaf}, and we found that $\mathcal{C}(\operatorname{Spin}(12))$ has the following features:

\begin{enumerate}[label=$\bullet$]
\item $3258$ extremal rays (minimal generating set over $\Q_{\ge 0}$)
\item $3470$ Hilbert basis elements (minimal generating set over $\Z_{\ge 0}$)
\item $28$ Hilbert basis elements not lying on a regular facet
\end{enumerate}

The computation finished in $16$ hours on the supercomputer. Here are those $28$ Hilbert basis elements not lying on a regular facet; they are listed up to $S_3$-action permuting the entries and $S_2$-action swapping indices $5$ and $6$ (i.e., the non-trivial Dynkin automorphism):

$$
\begin{array}{c}
(\omega_4,\omega_4,\omega_4)\\
(\omega_4+\omega_6,\omega_3+\omega_5,\omega_2+\omega_4)\\
%(\omega_4+\omega_5,\omega_3+\omega_6,\omega_2+\omega_4)\\
(2\omega_4,\omega_3+\omega_5+\omega_6,\omega_2+\omega_4)\\
(2\omega_4,\omega_3+\omega_5+\omega_6,2\omega_2+\omega_4)\\
(\omega_3+\omega_5+\omega_6,2\omega_3,2\omega_3)\\
\end{array}
$$

The possible $L^{ss}$ subgroups arising from regular facets are of the following types: 
$D_5$, $A_1\times D_4$, $A_2\times A_3$, $A_3\times A_1\times A_1$, and $A_5$.  It is known that the saturation conjecture holds for each of these (see \cite{KKM}, \cite{KT}, and part (a) of this theorem), so each Hilbert basis element $\vec\lambda$ of $\mathcal{C}(\operatorname{Spin}(12))$, except possibly the $28$ mentioned above, is in $\mathcal{R}(\operatorname{Spin}(12))$. 

Finally, we checked explicitly that each of the $28$ Hilbert basis elements not lying on a regular facet is an element of $\mathcal{R}(\operatorname{Spin}(12))$. We used the freely available software {\tt LiE} \cite{LiE} to accomplish this. In corroboration, we also checked that each of the $3470$ Hilbert basis elements does indeed lie in $\mathcal{R}(\operatorname{Spin}(12))$, for which we also used {\tt LiE}. This shows $\mathcal{C}(\operatorname{Spin}(12))\subseteq \mathcal{R}(\operatorname{Spin}(12))$, and the result follows. 

\subsection{Remarks on the computational method}

We did, along the way, find $12144$ inequalities governing the cone $\mathcal{C}(\op{Spin}(12))$. However,
we found it is not computationally feasible to use {\tt Normaliz} to find the Hilbert basis from these inequalities instead of from the rays mentioned above. That is to say, we use the formulas (\ref{tI}) and (\ref{tII}) of \cite{BKiers} in a crucial way. 

According to the documentation, {\tt Normaliz} presents two algorithms for obtaining the Hilbert basis of a cone: the ``primal'' algorithm and the ``dual'' algorithm. Generally speaking, the primal algorithm is optimal if the input description of the cone is via generators (rays) and the dual algorithm is optimal if the input description of the cone is via constraints (inequalities), but there are exceptions. Some routine time tests on types $A_1,A_2,A_3,A_4,A_5,D_4$, and $D_5$ show that actually the primal algorithm is far superior for our usage here, where input was via constraints. 

Further time tests on types $A_1,A_2,A_3,A_4,A_5,D_4$, and $D_5$ show that the primal algorithm used on the (very redundant) generators coming from (\ref{tI}) and (\ref{tII}) is far superior to using the primal algorithm with constraints as inputs. Also indicative of this, of course, is the case of type $D_6$: the generators method was computable and the constraints method was not. 

At least in these low-rank situations, then, it seems the combination of {\tt Normaliz}'s primal algorithm on generators coming from formulas (\ref{tI}) and (\ref{tII}) of \cite{BKiers} is most efficient in determining the Hilbert basis of $\mathcal{C}(G)$.

%\subsection{Verification of method}

%The author did use the $D_6$ method (that is, use formulas (\ref{tI}) and (\ref{tII}) to obtain rays and then {\tt Normaliz} to obtain minimal rays and Hilbert basis) to reproduce all the results listed below relevant to types $A$ and $D$. 

\section{Related Results}

\subsection{The saturated tensor cones for type $A$ of small rank}

Using a computer (code written in {\tt Sage 8.0} \cite{sagemath}) and both the inequalities and rays procedures described above, the following results were obtained for $G = \text{SL}(n+1)$ (type $A_n$), for $n=1,2,3,4,5$. The total number of inequalities is expressed as $a+b$, where $a$ is the number of inequalities coming from (\ref{ineqs}) and $b$ is the number of chamber inequalities (always $3\times$rank). In each case below, there is also one equality that ensures the sum in the root lattice condition. \\

\begingroup
\setlength{\tabcolsep}{10pt} % Default value: 6pt
\renewcommand{\arraystretch}{1.5} % Default value: 1
\begin{center}
\begin{tabular}{|c|c|c|c|c|}\hline
rank & total ineqs. & H.b. elements & extremal rays & H.b. elements not on a regular facet \\\hline\hline
$1$ & $3+3$ & $3$ & $3$ & $0$ \\\hline
$2$ & $12+6$ & $8$ & $8$ & $0$ \\\hline
$3$ & $41+9$& $18$ & $18$ & $0$ \\\hline
$4$ & $142+12$ & $42$ & $42$ & $0$ \\\hline
$5$ & $521+15$ & $112$ & $112$ & $0$ \\\hline
\end{tabular}
\end{center}~\\
\endgroup

The counts of inequalities for ranks $2$ and $3$ agree the results listed in \cite{Kumar} and \cite{KLM}. The number of extremal rays for rank $2$ agrees with \cite{KaLM}.

%\footnotetext[1]{cf. the $12$ non-chamber inequalities listed in \cite{Kumar}}
%\footnotetext[2]{cf. the $8$ rays listed in \cite{KaLM}}
%\footnotetext[3]{cf. the $50$ inequalities mentioned in \cite{KLM} (or the $41$ non-chamber inequalities listed in \cite{Kumar})}

\subsection{The saturated tensor cones for type $C$ of small rank}

In the same fashion, the following results were obtained for $G = \text{Sp}(2n)$ (type $C_n$), for $n=2,3,4,5$. In each case below, there is also one equality that ensures the sum in the root lattice condition. \\

\begingroup
\setlength{\tabcolsep}{10pt} % Default value: 6pt
\renewcommand{\arraystretch}{1.5} % Default value: 1
\begin{center}
\begin{tabular}{|c|c|c|c|c|}\hline
rank & total ineqs. & H.b. elements & extremal rays & H.b. elements not on a regular facet \\\hline\hline
$2$ & $18+6$& $13$ & $12$ & $1$ \\\hline
$3$ & $93+9$ & $58$ & $51$ & $1$ \\\hline
$4$ & $474+12$ & $302$ & $237$ & $2$ \\\hline
$5$ & $2421+15$ & $1598$ & $1122$ & $16$ \\\hline
\end{tabular}
\end{center}~\\
\endgroup

The results for ranks $2$ and $3$ above agree with those found in \cite{KaLM}, \cite{Kumar}, and \cite{KLM}. 

%\footnotetext[4]{cf. the $18$ (non-chamber) inequalities calculated in \cite{KaLM} (see also \cite{Kumar})}
%\footnotetext[5]{cf. the $12$ rays listed in \cite{KaLM}}
%\footnotetext[6]{cf. the $102$ inequalities mentioned in \cite{KLM} (or the $93$ non-chamber inequalities listed in \cite{Kumar})}
%\footnotetext[7]{cf. the $51$ rays listed in \cite{KLM}}

\begin{remark}
The saturated tensor cones for type $C$ and type $B$ are isomorphic due to the duality at the level of root systems, so the above data may be interpreted as results for type $B$ as well. 
\end{remark}

 \begin{remark}
It is known that the saturation conjecture fails for the aforementioned cones. For each of $n=2,3,4,5$, we verified this fact by finding Hilbert basis elements which fail to lie in $\mathcal{R}(\text{Sp}(2n))$. 
 \end{remark}

\subsection{Summary of results for type $D$}

The following table summarizes known features of the cones for type $D$ of small rank (starting at rank $4$):
~\\

\begingroup
\setlength{\tabcolsep}{10pt} % Default value: 6pt
\renewcommand{\arraystretch}{1.5} % Default value: 1
\begin{center}
\begin{tabular}{|c|c|c|c|c|}\hline
rank & total ineqs. & H.b. elements & extremal rays & H.b. elements not on a regular facet \\\hline\hline
$4$ & $294+12$ & $82$ & $81$ & $1$ \\\hline
$5$ & $1967+15$ & $505$ & $492$ & $0$ \\\hline
$6$ & $12144+18 $ & 3470 & 3258 & $28$ \\\hline
\end{tabular}
\end{center}~\\
\endgroup

The results for rank $4$ agree with the $306$ inequalities, $82$ H.b. elements, and $81$ extremal rays given in \cite{KKM}. 

The only Hilbert basis element for $\mathcal{C}(\text{Spin}(8))$ which does not lie on a regular facet is $(\omega_2,\omega_2,\omega_2)$. Because $\omega_{n-2}$ is self-dual for type $D_n$, $n$ even, the element $(\omega_{n-2},\omega_{n-2},\omega_{n-2})$ will always be a Hilbert basis element. Prior to obtaining the Hilbert basis and extremal rays for type $D_6$, we checked directly that $(\omega_4,\omega_4,\omega_4)$ does not lie on any regular facet for type $D_6$ as well. We na\"ively asked the following 
\begin{question}\label{bigQ}
Let $G$ be simple, simply-connected of type $D_n$.

For $n$ even: is $(\omega_{n-2},\omega_{n-2},\omega_{n-2})$ the only Hilbert basis element of $\mathcal{C}(G)$ not lying on a regular facet? 

For $n$ odd: are there never Hilbert basis elements of $\mathcal{C}(G)$ not lying on a regular facet? 
\end{question} 

The answer to this question is negative, at least for even $n$. There were indeed $27$ other Hilbert basis elements of $\mathcal{C}(\op{Spin}(12))$ not lying on a regular facet. 

\subsection{``Non-Fultonian'' Hilbert basis elements in $\mathcal{C}(\operatorname{Spin}(10))$}

\begin{definition}
Say a triple $(\lambda_1,\lambda_2,\lambda_3)\in \mathcal{C}(G)$ has the Fulton scaling property if 
$$
\dim (V(N\lambda_1)\otimes V(N\lambda_2)\otimes V(N\lambda_3))^G = 1,
$$
for every $N\ge 1$. Call such a triple ``Fultonian'' for short. 
\end{definition}

In type $A$, it is known that 
\begin{align}\label{FC}
\dim(V(\lambda_1)\otimes V(\lambda_2)\otimes V(\lambda_3))^G = 1 \implies (\lambda_1, \lambda_2,\lambda_3) \text{ is Fultonian.} \end{align}
This was conjectured by Fulton - hence the name - and first proved by Knutson-Tao-Woodward \cite{KTW}. The direct generalization of this conjecture for arbitrary $G$ does not hold; this implies that some cones $\mathcal{C}(G)$ contain non-Fultonian elements. We list here certain elements of $\mathcal{C}(\operatorname{Spin}(10))$ which are non-Fultonian, including some which cause the implication (\ref{FC}) to fail. All claims were verified using {\tt LiE} \cite{LiE}. 

 \subsubsection{} The following $13$ Hilbert basis elements $(\lambda_1,\lambda_2,\lambda_3)$ satisfy 
 \begin{align}\label{whoknows}
 \dim  (V(N\lambda_1)\otimes V(N\lambda_2)\otimes V(N\lambda_3))^G = \left\lfloor \frac{N}{2} \right\rfloor +1,
 \end{align}
 for $N = 1,2,\hdots,5$. They are listed only up to permutation:
\begin{center}
\begin{tabular}{cc}
$(\omega_2,\omega_3,\omega_3)$
&$(\omega_1+\omega_3,\omega_3,\omega_3)$\\
$(\omega_2,\omega_2,\omega_2)$&
$(2\omega_2,2\omega_3,\omega_2+\omega_4+\omega_5)$.
\end{tabular}
\end{center}

Therefore each of these Hilbert basis elements is non-Fultonian and, furthermore, fails implication (\ref{FC}). Interestingly, these $13$ Hilbert basis elements are the same $13$ ($=505-492$) which are not extremal rays. It is not known whether formula (\ref{whoknows}) holds for all $N\ge 1$ for these elements. 

\subsubsection{} The following $3$ Hilbert basis elements $(\lambda_1,\lambda_2,\lambda_3)$ satisfy 
\begin{align}\label{whoknows2}
 \dim  (V(N\lambda_1)\otimes V(N\lambda_2)\otimes V(N\lambda_3))^G = N+1,
 \end{align}
 for $N = 1,2,\hdots,5$. They are the $3$ permutations of the single element
 $$
 (\omega_3,\omega_3,\omega_4+\omega_5).
 $$
 Therefore each of these is non-Fultonian. It is not known whether formula (\ref{whoknows2}) holds in general for these three. 
 
 These Hilbert basis elements also give extremal rays of the cone. As discussed in \cite{BKiers}, all extremal rays of $\mathcal{C}(G)$ lie on a facet $\mathcal{F}$. The extremal rays on a facet may be classified as either ``Type I'' or ``Type II''; however, it is possible for rays to be Type I on one facet and Type II on another. Because every Type I ray is Fultonian, the three aforementioned Hilbert basis elements give examples of extremal rays which are not Type I on any facet. 

\appendix

\section{Another proof of Proposition \ref{orange}}
Directly from the definition of $\Psi$, one may deduce the properties in the proposition as follows. 
For a fixed $w\ne w_0$, write $w^{-1}w_0 = \sigma_\gamma\sigma_{\gamma_2}\cdots\sigma_{\gamma_t}$ as a reduced word. Then $\tilde P_w = A_{\gamma}Q$, where $Q = A_{\sigma_{\gamma_2}\cdots\sigma_{\gamma_t}}\tilde P_{w_0}$. In particular, $A_\gamma \tilde P_w = 0$. Now let $W'$ be a set of representatives for the cosets $W/\langle \sigma_\gamma\rangle$. We can therefore write
\begin{align*}
\frac{1}{\prod_{\alpha\in \Phi^+} \alpha} \sum_{\sigma \in W} (-1)^{\ell(\sigma)} \sigma(\tilde P_w)& = \frac{1}{\prod_{\alpha\in \Phi^+}\alpha} \sum_{\sigma\in W'} \left(
(-1)^{\ell(\sigma)} \sigma(\tilde P_w) - (-1)^{\ell(\sigma)} \sigma\sigma_\gamma(\tilde P_w)\right)\\
&=\frac{1}{\prod_{\alpha\in \Phi^+}\alpha} \sum_{\sigma\in W'} (-1)^{\ell(\sigma)} \sigma\left(\frac{\tilde P_w-\sigma_\gamma\tilde P_w}{\gamma}\right)\sigma(\gamma)\\
&= \frac{1}{\prod_{\alpha\in \Phi^+}\alpha} \sum_{\sigma\in W'} (-1)^{\ell(\sigma)}\sigma(\gamma)\sigma\left(A_\gamma \tilde P_w\right)\\
&= 0.
\end{align*}
Furthermore, one easily checks that 
\begin{align*}
\sum_{\sigma\in W} \frac{\sigma(\tilde P_{w_0})}{\sigma(\prod_{\alpha\in \Phi^+}\alpha)}&=1.
\end{align*}

Since the expression $\sum_{\sigma\in W} \frac{\sigma(f)}{\sigma(\prod_{\alpha\in \Phi^+}\alpha)}$ is linear in $f=\sum \tilde P_w f_w$, $\Psi(f)$ is well-defined and equals $f_{w_0}(0)$ (here we use that the $f_w$ are $W$-invariant). If $f$ is of degree $\le \deg \tilde P_{w_0}$, then $f_{w_0}$ can be assumed constant. In such a case, 
$$
\frac{1}{\prod_{\alpha\in \Phi^+} \alpha} \sum_{\sigma \in W} (-1)^{\ell(\sigma)} \sigma(f) = f_{w_0},
$$
and evaluating at any point of $\mathfrak{h}_\Q$ gives $\Psi(f)=f_{w_0}$, which is also the $P_{w_0}$-coefficient of $\bar f\in R/J$.

\vspace{0.05in}

\noindent
Department of Mathematics, University of North Carolina, Chapel Hill, NC 27599\\
{{email: jokiers@live.unc.edu}}

\end{document}